\documentclass[a4paper]{article}
\usepackage{amsthm}
\usepackage{hyperref}
\hypersetup{colorlinks=true,linkcolor=blue,citecolor=red,linktocpage=true}
\newtheorem{proposition}{Proposition}[section]
\newtheorem{lemma}[proposition]{Lemma}
\newtheorem{theorem}[proposition]{Theorem}
\newtheorem{corollary}[proposition]{Corollary}

\newtheorem{THM}{Theorem}

\theoremstyle{definition}
 \newtheorem{definition}[proposition]{Definition}
 \newtheorem{example}[proposition]{Example}
 
\usepackage[utf8]{inputenc}
\usepackage[english]{babel}
\usepackage{amsmath}
\usepackage{amssymb}
\usepackage{textcomp}
\usepackage{hyphenat}
\usepackage[all]{xy}
\usepackage{graphicx}
\usepackage{float}
\usepackage{listofsymbols}
\usepackage{tikz}
\usepackage{circuitikz}
\usepackage{tikz-cd}

\usepackage{latexsym}
\usepackage{epsfig}

\newcommand{\C}{\mathbb C}
\newcommand{\Z}{\mathbb{Z}}
\newcommand{\F}{\mathcal{F}}
\newcommand{\bigslant}[2]{{\raisebox{.2em}{$#1$}\left/\raisebox{-.2em}{$#2$}\right.}}

\usepackage{nomencl}
\makenomenclature

\emergencystretch=\maxdimen
\begin{document}

\title{Homotopy groups of generic leaves of logarithmic foliations.}
\author{Diego Rodr\'{i}guez\footnote{The research for this paper is based on the author's thesis presented for the degree of Doctor in Mathematics at IMPA, July 2016. I was supported by CONACYT'-FORDECYT 265667 to conclude this paper.}}
\date{November 2017}
\maketitle

\begin{abstract}
We study the homotopy groups of generic leaves of logarithmic foliations on complex projective manifolds. We exhibit a relation between the homotopy groups of a generic leaf and of the complement of the polar divisor of the logarithmic foliation.
\end{abstract}
\section{Introduction}
A logarithmic foliation $\F$ on a complex projective manifold $X$ is defined by a closed logarithmic 1-form $\omega$ with polar locus $D=\sum_j D_j$, with $D_j$ irreducible hypersurface of $X$. Here $\mathcal{L}$ denotes a non singular leaf of $\F$, which is an immersed complex manifold of codimension one in $X$. In general, $\mathcal{L}$ is a transcendental leaf, that is, it is not contained in any projective hypersurface of $X$. 

We will consider the following topological properties of complex projective manifolds:
\begin{itemize}
\item[(i)] If $X$ is a smooth hypersurface of the projective space $\mathbb{P}^{n+1}$ with $n>1$, then $X$ is simply connected. 
\end{itemize}
Also, if a complex projective manifold $X$, with dimension $n$, lies in $\mathbb{P}^m$ and $H\subset \mathbb{P}^m$ is a general hyperplane, the Lefschetz hyperplane section theorem, implies that the following assertions hold:
 \begin{itemize}
 \item[(ii)] If $n>1$ then the hyperplane section $X\cap H$ is connected.
 \item[(iii)] If $n>2$ then the fundamental groups of $X$ and $X\cap H$ are isomorphic.
\end{itemize}  

Based on the statements above, Dominique Cerveau in \cite[Section 2.10]{cerveau2013quelques} proposes to study which conditions over a logarithmic foliation $\F$ on a complex projective manifold are sufficient for a generic leaf $\mathcal{L}$ of $\F$ to satisfy these topological properties. 

In this article, we shall study this issue using Homotopy Theory. Our first result exhibits a relation between the homotopy groups of a generic leaf $\mathcal{L}$ and of the complement $X-D$ of the polar divisor $D$ of the closed logarithmic 1-form $\omega$ defining the foliation $\F$ on a complex projective manifold $X$.

\begin{THM}\label{Theorem:genericLeafhomotopy}
Let $\mathcal{F},\omega,D,X$ satisfy the above assumptions, with $\mathrm{dim}_\C X=n+1$ and $n>1$. If $D$ is a simple normal crossing ample divisor, then the fundamental group of $\mathcal{L}$ is isomorphic to the group
\[
G:=\left\{[\gamma]\in\pi_1(X-D)|\int_\gamma \omega=0\right\},
\]
where $\gamma$ is a closed curve in $X-D$. Furthermore, the morphisms of homotopy groups 
\[ i_*:\pi_l(\mathcal{L})\rightarrow\pi_l(X-D),\]
induced by the inclusion $i:\mathcal{L}\hookrightarrow X-D$ are isomorphisms if $1<l<n$ and epimorphisms if $l=n$.
\end{THM}

If $X$ is the projective space $\mathbb{P}^{n+1}$, we prove the following Lefschetz hyperplane section type theorem. 

\begin{THM}\label{theorem:LefschetzType}
Let  $\mathcal{F}$ be a logarithmic foliation defined by a logarithmic $1$-form  $\omega$ on $\mathbb{P}^{n+1}$, $n\ge 1$,
with  a simple normal crossing  polar divisor  $D$. Let $H\subset\mathbb{P}^{n+1}$ be a hyperplane such that $H\cap D$
is a reduced divisor with simple normal crossings in $H$. Suppose the leaves $\mathcal{L},\mathcal L \cap H$ are generic leaves of $\mathcal F,\mathcal{F}|_H$ respectively. Then the morphism between homotopy groups
\[
(\mathit{i})_*:\pi_l(\mathcal{L}\cap H)\rightarrow \pi_l(\mathcal{L}),
\]
induced by the inclusion $\mathit{i}:\mathcal{L}\cap H\hookrightarrow\mathcal{L}$ is
\begin{itemize}
\item[a)] an isomorphism if  $l<n-1$,
\item[b)] an epimorphism  if $l=n-1$.
\end{itemize}
\end{THM}

This result shows that the claims (ii,iii) are true for generic leaves of logarithmic foliations on $\mathbb{P}^{n+1}$. For generic logarithmic foliations on $\mathbb{P}^{n+1}$, Theorem \ref{Theorem:genericLeafhomotopy}  implies that this foliations have simply connected generic leaves.

In order to prove these results we shall adapt a result of Carlos Simpson \cite[Corollary 21]{simpson1993lefschetz}, which concerns the topology of integral varieties of a closed holomorphic 1-form on a projective variety.

\section{Generic leaves of logarithmic foliations}
Let $\omega$ be a closed logarithmic 1-form on a complex projective manifold $X$ with polar divisor $D=\sum D_i$. Note that for any point $p\in X$ there is a neighborhood $U$ of $p$ in $X$ such that $\omega|_U$ can be written as
\begin{equation}\label{form:Local-l1f}
\omega_{0}+\sum_{j=1}^{r}\lambda_{j}\frac{df_{j}}{f_{j}},
\end{equation} 
 where $\omega_0$ is a closed holomorphic 1-form on $U$, $\lambda_j\in\C^{\ast}$ and $f_j\in\mathcal{O}(U)$, and $\{f_{j}=0\}$, $j=1,\ldots,r$, are the reduced equations of the irreducible components of $D\cap U$	. 
 
 If the polar divisor $D\subset X$ is simple normal crossing, then each irreducible component of $D$ is smooth and locally near of each point $D$ can be represented in a chart $(x_i):U\rightarrow M$ as the locus $\{x_0\cdots x_k=0\}$ with $k+1\leq \mathrm{dim}(X)$. Moreover, there is a coordinate chart $(U,(y_j))$ for each point $q\in X$ such that $\omega$ can be written as
 \begin{equation}\label{formula:log}
((y_j)^{-1})^*\omega=\sum_{j=0}^{k}\lambda_j\frac{dy_j}{y_j}.
\end{equation}

Consider a singular point $p$ of $\omega$ not in $\mathrm{Sing}(D)$. Since $D$ is simple normal crossing, (\ref{formula:log}) shows that the connected component $S_p$ of $\{x\in X|\omega(x)=0\}$, which contains $p$, has empty intersection with $D$ (see \cite[Theorem 3]{cukierman2006singularities} for more details). When $D$ is an ample divisor in $X$, then any complex subvariety of dimension greater than zero has nonempty intersection with $D$. In particular, we have the following result.

\begin{lemma}\label{corollary:isolatedpoint}
Let $\omega,p,S_p$ be as above. If $D$ is a normal crossing ample polar divisor of $\omega$, then $S_p$ is the isolated point $p$.
\end{lemma}

A closed logarithmic 1-form $\omega$ on $X$ with polar divisor $D$ defines the following group morphism
\begin{equation}\label{function:phi}
\begin{array}{lll}
\phi:\pi_1(X-D) & \rightarrow & (\C,+) \\
\qquad\qquad [\gamma]  & \mapsto & \int_\gamma \omega,
\end{array}
\end{equation}
where $\gamma$ is a closed curve in $X-D$. Consider a normal subgroup $G$ of $\pi_1(X-D)$, which is contained in $\ker\phi$. The group $G$ defines a regular covering space $ \rho:Y\rightarrow X-D$ with the property that $\rho_*(\pi_1(Y))=G$. Taking a closed curve $\eta:I\rightarrow Y$ we see that $[\rho\circ\eta]\in\ker\phi$. We thus get

\[
\int_\eta\rho^{*}\omega=\int_{\rho\circ\eta}\omega=0
.\]
Consequently, for a fixed point $y_0\in Y$ the function
\begin{equation}\label{equation:primitive}
g(y)=\int_{y_0}^{y}\rho^{*}\omega
\end{equation}
is well defined for $y\in Y$. In this way, we obtain what will be referred to as a \emph{$\omega$-exact covering space} of $X-D$.

If the kernel of $\phi$ is a non trivial group, then there are at least two \mbox{$\omega$-exact} covering spaces of $X-D$: the universal cover and the regular cover \mbox{$\rho:Y\rightarrow X-D$} such that $\rho_*(\pi_1(Y))=\ker\phi$.

The following theorem is an adaptation of \cite[Corollary 21]{simpson1993lefschetz}, which allows us to exhibit relations between the homotopy groups of a cover of an integral manifold of $\omega$ and the regular cover of the complement of the polar divisor of $\omega$.

\begin{theorem}[\bf Lefschetz-Simpson Theorem]\label{theorem:integralleaves}
Let $\omega$ be a closed logarithmic 1-form on a projective manifold $X$ of dimension $n+1$, $n\geq 1$. Assume that the polar divisor $D$ of $\omega$ is simple normal crossing ample divisor. Then for a \mbox{$\omega$-exact} covering space $\rho:Y\rightarrow X-D$ and a function $g$ (\ref{equation:primitive}), the pair $(Y,g^{-1}(c))$ is $n$-connected for any $c\in\C$.
\end{theorem}

The $n$-connectedness of the pair $(Y,g^{-1}(c))$ is equivalent to the morphisms of homotopy groups
\[
\pi_i(g^{-1}(c))\rightarrow \pi_i(Y),
\]
induced by the inclusion $g^{-1}(c)\hookrightarrow Y$ being isomorphisms if $i<n$ and epimorphism for $i=n$.

The statement above will be proved in the next section. We will use this theorem and standard techniques of Homotopy Theory to prove the Theorems \ref{Theorem:genericLeafhomotopy} and \ref{theorem:LefschetzType}.

Now, we study the homotopy groups of generic leaves $\mathcal{L}$ of $\F$. The following relation between homotopy groups of a path connected topological space $Y$ and its regular covering space $\rho:\tilde{Y}\rightarrow Y$ will be very useful.

\textit{The covering space projection} $\rho$ \textit{induces isomorphisms} \[\rho_{*}:\pi_i(\tilde{Y},\tilde{y}_0)\rightarrow\pi_i(Y,y_0),\] \textit{with} $\tilde{y}_0\in \rho^{-1}(y_0)$, \textit{between homotopy groups of dimension greater than} 1.

\begin{definition}
Let $\F$ be a logarithmic foliation on $X$ defined by a closed logarithmic 1-form $\omega$ with polar divisor $D$. Let $\rho:Y\rightarrow X-D$ and $g:Y\rightarrow\mathbb{C}$ be as above. We will say that a leaf $\mathcal{L}$ of $\F$ is \emph{generic} if for some component $C$ of $\rho^{-1}(\mathcal{L})$ the value of $g$ on $C$ is a regular value of $g$.
\end{definition}

\begin{proposition}\label{Proposition:exactcover}
Under the conditions stated above, we assume that the polar divisor $D$ of $\omega$ is a simple normal crossing ample divisor. Then there exists a $\omega$-exact covering space $\rho:Y\rightarrow X-D$ such that for a primitive $g$ of $\rho^*\omega$ defined by (\ref{equation:primitive}) the inverse image $g^{-1}(c)$ of a regular value $c\in\C$ is biholomorphic to a generic leaf $\mathcal{L}$ of $\F$.
\end{proposition}
\begin{proof}
Consider the $\omega$-exact covering space $\rho:Y\rightarrow X-D$ satisfying $\rho_*(\pi_1(Y))=\ker\phi$, with $\phi$ defined by (\ref{function:phi}). Write $G=\ker\phi$.

 We will show that for a regular value $c\in\C$ of $g$ the restriction  \[\rho|_{g^{-1}(c)}:g^{-1}(c)\rightarrow\mathcal{L}\] is a  biholomorphism.

Suppose not. Then there exist distinct points  $y_0,y_1\in\rho^{-1}(x_0)$, with $x_0\in\mathcal{L}$, such that $y_0,y_1\in g^{-1}(c)$.

Take $\tilde{\gamma}:I\rightarrow Y$ with $\tilde{\gamma} (0)=y_0$ y $\tilde{\gamma} (1)=y_1$. Since $n\geq1$, Theorem \ref{theorem:integralleaves} implies
that the pair $(Y,g^{-1}(c))$ is 1-connected. This implies that
there exists  $\gamma'$ contained in $g^{-1}(c)$ homotopic to $\tilde{\gamma}$ with fixed endpoints. Therefore $\gamma=\rho\circ\gamma'$
is a curve in $\mathcal{L}$ which is not homotopically trivial in $X-D$. But since it is contained in a leaf of the foliation
 we have that
\[
\int_{\gamma}\omega=0.
\]
Hence $\gamma$ is homotopic to an element of  $G$ a contradiction. Thus  $\rho|_{g^{-1}(c)}$ is a biholomorphism. This is the desired conclusion.

\end{proof}

\begin{proof}[{\bf Proof of Theorem \ref{Theorem:genericLeafhomotopy}}]
Let $\rho:Y\rightarrow X-D$ be the $\omega$-exact cover given by Proposition \ref{Proposition:exactcover}. By Theorem \ref{theorem:integralleaves}, the morphisms
\[
((\rho|_{g^{-1}(c)})^{-1}\circ i)_*:\pi_l(\mathcal{L})\rightarrow\pi_l(Y)
\]
are isomorphisms if $l<n$ and epimorphisms if $l=n$, where the map $\rho|_{g^{-1}(c)}$ is the biholomorphism between $g^{-1}(c)$ and $\mathcal{L}$, and $i:\mathcal{L}\hookrightarrow X-D$ is the inclusion map. As $n$ is greater than 1 we have that $\pi_1(\mathcal{L})$ is isomorphic to
\[\pi_1(Y)\cong\left\{[\gamma]\in\pi_1(X-D)|\int_\gamma \omega=0\right\}.\]

Since $Y$ is a covering space of $X-D$, it follows that the morphisms $(\rho)_*$ between the groups $\pi_l(Y)$ and $\pi_l(X-D)$ are isomorphisms if $l>1$. Therefore the morphisms
\[ i_*:\pi_l(\mathcal{L})\rightarrow\pi_l(X-D)\]
are isomorphisms if $1<l<n$ and epimorphisms if $l=n$.
\end{proof}

We will use the statements in \cite{dimca2012singularities,Hamm1973,libgober2004homotopy} about the topology of the complement of a divisor in a projective manifold to establish the results below. 

 Let $H$ be an abelian free group generated by the components $D_i$ of a divisor $D\subset X$. If $X$ is simply connected and each irreducible component of the simple normal crossing divisor $D=\sum_{i=0}^{k}D_i$ is  ample, Corollary 2.2 of \cite{libgober2004homotopy} implies that the fundamental group of the complement $X-D$ is isomorphic to the cokernel of the morphism
\[
\begin{array}{lll}
h:H_2(X,\mathbb{Z}) & \rightarrow & H \\
\quad\quad a  & \mapsto & \sum_{i=0}^{k}(a,D_i)D_i,
\end{array}
\]
where $(a,D_i)$ is the Kronecker pairing, here we associate $D_i$ with its Chern class in $H^2(X,\mathbb{Z})$ (see \cite[p. 15]{davis2001lecture} for more details). If $X$ is the projective space $\mathbb{P}^{n+1}$, then the image of the morphism $h$ is generated by $(d_0,\ldots,d_k)$, where $d_i$ is the degree of $D_i$. In particular, we have the following result

\begin{corollary}\label{corollary:LeavesP}
Under the hypotheses of Theorem \ref{Theorem:genericLeafhomotopy}, if $X=\mathbb{P}^{n+1}$ then the fundamental group of a generic leaf $\mathcal{L}$ is isomorphic to the following group
\[
\bigslant{\left\{(m_0,\ldots,m_k)\in\mathbb{Z}^{k+1}|\sum_{i=0}^{k}\lambda_i m_i=0\right\}}
{\mathbb{Z}(d_0,\ldots,d_k)},
\]
where $\lambda_i=\mathrm{Res}(D_i,\omega)$ are the residues of $\omega$ around $D_i$.
\end{corollary}

\begin{example}
Let us consider the case where the polar divisor of the logarithmic $1$-form $\omega$ has only two irreducible components, say  $D_0$ and $D_1$.
If the degrees  $d_0,d_1$ are equal then the leaves  $\mathcal{L}$ of the foliation $\mathcal{F}$ are contained in elements of the
pencil
\[
\{aF_0+bF_1|(a:b)\in\mathbb{P}^1\}.
\]
In particular the generic leaf  $\mathcal{L}$ is of the form
$\{aF_0+bF_1=0\}-D$ for $(a:b)$ generic. The  Corollary \ref{corollary:LeavesP} implies that
\[\pi_1(\mathcal{L})=\dfrac{\Z}{d\Z},\]
which is the torsion subgroup of $\pi_1(\mathbb{P}^3-D)$. 
\end{example}

\begin{corollary}\label{corollary:completintersection}
Let $X^{n+1}$ be a complete intersection in $\mathbb{P}^N$. Let $D$ be an arrangement of  hyperplanes in $\mathbb{P}^N$ such that $D\cap X$ is simple normal crossing in $X$. If a logarithmic foliation $\mathcal{F}$ on $X$ has polar divisor $D\cap X$ then any generic leaf of $\mathcal{F}$ satisfies $\pi_l(\mathcal{L})=0$ for $1<l<n$.
\end{corollary}

\begin{proof}
From Theorem \ref{Theorem:genericLeafhomotopy} the homotopy groups of dimension $l$, with $1<l<n$, of a generic leaf are isomorphic to the respective homotopy groups of $X-D$. By \cite[Theorem 2.4]{libgober2004homotopy} the homotopy groups of $X-D$ of dimension $1<l<n$ are trivial, and the corollary follows. 
\end{proof}

Let $\omega$ be a closed logarithmic 1-form as in Corollary \ref{corollary:LeavesP}. The residues $\{\lambda_j\}_{j=0}^k$ of $\omega$ are non resonant if all integer solutions $(m_0,\ldots,m_k)\in\Z^{k+1}$ of the equation $\sum_{j=0}^k m_j\lambda_j=0$ are contained in $\Z(d_0',\ldots,d_k')$, where $d_j'\cdot\mathrm{gcd}(d_0,\ldots,d_k)=d_j$.

 If the residues $\{\lambda_j\}_{j=0}^k$ are  non resonant and $\mathrm{gcd}(d_0,\ldots,d_k)=1$, the Corollary \ref{corollary:LeavesP} implies that  the generic leaf $\mathcal{L}$ of the foliation $\mathcal{F}$ defined by $\omega$ has trivial fundamental group. Therefore we deduce from Corollary \ref{corollary:completintersection},

\begin{corollary}
Under the assumptions of Corollary \ref{corollary:LeavesP}, if moreover $d_j=1$ and the residues $\lambda_i$ are non resonant. Then the generic leaf  $\mathcal{L}$  of the foliation $\mathcal{F}$ is $(n-1)$-connected.
\end{corollary}

Next result gives a relation between the homotopy groups of a generic leaf of a logarithmic foliation on $\mathbb{P}^{n+1}$ and its general hyperplane sections.
\begin{proof}[{\bf Proof of Theorem \ref{theorem:LefschetzType}}]
Let $D(H)=H\cap D$. The inclusion $\mathit{i}$ from $H-D(H)$ to $\mathbb{P}^{n+1}-D$ induces the morphisms
\begin{equation}\label{morphism:homotopy}
\mathit{i}_*: \pi_l(H-D(H))\rightarrow \pi_l(\mathbb{P}^{n+1}-D)
\end{equation}
in homotopy. From the Lefschetz-Zariski type \cite[Theorem 0.2.1]{Hamm1973} we have that  $\mathit{i}_*$ is an  isomorphism for $l<n$ and an  epimorphism for $l=n$.

Consider the regular cover  $\rho:Y\rightarrow\mathbb{P}^{n+1}-D$ given by Proposition \ref{Proposition:exactcover}.
Let  $g$ be a primitive of  $\rho^{*}\omega$. Let $Y(H)=\rho^{-1}(H-D(H))$. Notice that $Y(H)$ is a connected regular covering space of $H-D(H)$. Let $g_H$
be the restriction of  $g$ to $Y(H)$. Let $c\in\C$ be a regular value of $g$ and $g_H$. Since $l\leq n-1$, Theorem
 \ref{theorem:integralleaves} implies that the morphisms
\[
\mathit{i}_*:\pi_l(g^{-1}(c))\rightarrow \pi_l(Y)\quad
\text{and}\quad
\mathit{i}_*:\pi_l(g_H^{-1}(c))\rightarrow \pi_l(Y(H))
\]
induced by the inclusion $Y(H)\hookrightarrow Y$, are  isomorphisms if $l<n-1$ and epimorphisms  if $l=n-1$.
Considering the long exact sequence of homotopy groups we obtain the following commutative diagram
for $l>0$:
\[
\xymatrix{
\cdots \ar[r]  &
\pi_{l+1}(Y(H),g^{-1}_H(c)) \ar[r] \ar[d] &
\pi_{l}(g^{-1}_H(c))\ar[r] \ar[d] &
\pi_{l}(Y(H)) \ar[r] \ar[d] &
\cdots \\
\cdots \ar[r] &
\pi_{l+1}(Y,g^{-1}(c)) \ar[r] &
\pi_{l}(g^{-1}(c)) \ar[r] &
\pi_{l}(Y) \ar[r] &
\cdots
}
\]
Since $Y,Y(H)$ are covers of $\mathbb{P}^{n+1}-D,H-D(H)$ respectively, the morphism  (\ref{morphism:homotopy}) shows that the morphisms
\[
\pi_{l}(Y(H)) \rightarrow\pi_{l}(Y)
\]
are isomorphisms for $l<n$ and epimorphisms for  $l=n$.
Analogously, Theorem \ref{theorem:integralleaves} implies that the morphisms
$\pi_{l}(Y(H),g^{-1}_H(c)) \rightarrow\pi_{l}(Y,g^{-1}(c))$
are isomorphisms for  $l<n$ and epimorphisms for $l=n$.  Applying the five Lemma, we have that the morphisms
$
\pi_{l}(g_H^{-1}(c)) \rightarrow\pi_{l}(g^{-1}(c))
$
are isomorphisms for $l<n-1$ and  epimorphisms for  $l=n-1$.
Hence the theorem follows from the biholomorphism given by Proposition \ref{Proposition:exactcover}.
\end{proof}

This verifies the claim (iii) for generic leaves of logarithmic foliations on projective spaces.

\section{Lefschetz-Simpson Theorem}

In order to prove the Theorem \ref{theorem:LefschetzType} we adapt the proof of \cite[Theorem 1]{simpson1993lefschetz} to the case of logarithmic closed 1-forms with a simple normal crossing ample polar divisor.  One of the key steps in the proof of Theorem \ref{theorem:integralleaves} consists in establishing an Ehresmann type result for the function $g$ outside an open neighborhood of the singular locus of $\rho^*\omega$. Before proving Theorem \ref{theorem:integralleaves} we will need to introduce some notation and to establish some preliminary results. Also, we follow the notation used in \cite{simpson1993lefschetz}.

We will use some properties in \cite{hatcher2002algebraic,simpson1993lefschetz} about Homotopy Theory to establish the statements bellow.

\subsection*{Singular Theory}

Under the hypotheses of Theorem \ref{theorem:integralleaves}, the Lemma \ref{corollary:isolatedpoint} implies that the singular locus of $\omega$ in $X-D$ is a finite union of isolated points.
Let  $\{p_i\}$ be the finite set of isolated singularities of $\omega$ in  $X-D$.
Fix a metric $\mu$ on $X$. Since $X$ is compact, $\mu$ is complete.
We can choose $\varepsilon_1>0$ sufficiently small such that the closed balls
\begin{equation*}
{B_\mu(p_i,\varepsilon_1)}=M_i
\end{equation*} are pairwise disjoint and the restriction of $\omega$ to an open neighborhood of
$M_i$ is exact. We define primitives  $g_i(x)=\int_{p_i}^{x}\omega$ for $x\in M_i$. Since the points  $p_i$ are isolated singularities, it follows from  \cite[Theorems 4.8, 5.10]{milnor2016singular} the existence of $\varepsilon_2>0$ sufficiently small such that
\begin{itemize}
\item[(i)]$0\in B(0,\varepsilon_2)\subset\C$ is the unique critical value for the primitive $g_i$;
\item[(ii)] the intersections $g_{i}^{-1}(0)\cap \partial M_i$ and $\quad g_{i}^{-1}(B(0,\varepsilon_2))\cap \partial M_i=T_i$ are smooth, and the
restriction of $\omega$ to $T_i$ is a $1$-form on  $T_i$ which never vanishes.
\end{itemize}

\begin{lemma}\label{lemma:singularTheory}
 Let $F_i=g_{i}^{-1}(0)$ and $E_i=g_{i}^{-1}(c)$ with $c\in B(0,\varepsilon_2)-\{0\}$ be fibers of $g_i$ restricted to
  \[N_i = M_i \cap g_i^{-1}(B(0,\varepsilon_2)).\] For small $\varepsilon_2$ the pair  $(N_i,F_i)$ is  $l$-connected
  for every  $l\in\mathbb{N}$ and the pair $(N_i,E_i)$ is $n$-connected.
\end{lemma}
\begin{proof}
For  $\varepsilon_2$ sufficiently small  \cite[Theorem 5.2]{milnor2016singular} implies that  $F_i$ is a deformation retract
of $N_i$. Therefore the pair  $(N_i,F_i)$ is $l$-connected for any $l\in\mathbb{N}$.

We know from   \cite[Theorems 5.11, 6.5]{milnor2016singular} that  $E_i$ has the homotopy type of a bouquet of spheres  $\mathbb{S}^{n}\vee\cdots\vee\mathbb{S}^{n}$ for $\varepsilon_2$ sufficiently small.
 Thus the fiber $E_i$ is $(n-1)$-connected. Since the neighborhood $N_i$ can be contracted to $p_i$
 the long exact sequence of Homotopy Theory implies that the pair $(N_i,E_i)$ is $n$-connected.
\end{proof}

\subsection*{Ehresmann type result}

Let $\rho:Y\rightarrow X-D$ be a $\omega$-exact covering space and the function $g$ a primitive of $\rho^*\omega$.
We will use  $j\in J_i$ as an index set for the points  $\tilde{p}_j$ of the discrete set $\rho^{-1}(p_i)$ and we will denote the union  $\cup J_i$ by $J$.
Fix  $\varepsilon_1,\varepsilon_2>0$ sufficiently small such that
\begin{itemize}
\item[(*)] in each connected component  $\tilde{M}_j$ of $\rho^{-1}(M_i)$ containing the point $\tilde{p}_j$, the restriction of  $\rho$ in $\tilde{M}_j$ is a biholomorphism; and
\item[(*)] the subsets  $N_i,T_i,F_i,E_i$ and the function $g_i$ satisfy the properties above for every $i$.
\end{itemize}
We define a primitive  $\tilde{g}_j=g_i\circ\rho$ for the restriction of  $\rho^{*}\omega$ to $\tilde{M}_j$ such that
 $g|_{\tilde{M}_j}=\tilde{g}_{j}+a_j$ for some  $a_j\in\C$, with $j\in J_i$. The subsets $\tilde{N}_j,\tilde{T}_j,\tilde{F}_j,\tilde{E}_j$ of $\tilde{M}_j$
 are the analogues of the subsets  $N_i,T_i,F_i,E_i$ of $M_i$.

We choose  $\delta$ such that  $0<5\delta<\varepsilon_2$. For each $b\in\C$, we define the subset $J(b)$ of $J$ formed by the indexes  $j$ such that  $|b-a_j|<3\delta$. Let $U_b=B(b,\delta)\subset \C$ and define the open subset of the covering space  $Y$
\[
W(b)=g^{-1}(U_b)\cap\left(\bigcup_{j\in J(b)}\tilde{N}_{j}^{\circ}\right)),
\]
where $\tilde{N}_{j}^{\circ}$ denotes the interior of $\tilde{N}_{j}$, which satisfies $
(g^{-1}(U_b)-W(b))\cap\overline{W(b)}$ is contained in $ \bigcup_{j\in J(b)}\tilde{T}_{j}.
$

We can now formulate our Ehresmann type result.

\begin{proposition}\label{proposition:locallyTrivial}
There  exists a trivialization of $g^{-1}(U_b)-W(b)$ with trivializing diffeomorphism
\[
\Phi:U_b\times (g^{-1}(b)-W(b))\rightarrow g^{-1}(U_b)-W(b),
\]
such that the restriction to the boundary satisfies
\[
\Phi (U_b\times (g^{-1}(b)-W(b))\cap\overline{W(b)})=(g^{-1}(U_b)-W(b))\cap\overline{W(b)}.
\]
\end{proposition}

\begin{proof}\label{SS:proof}
For each point  $q$ in the polar divisor $D\subset X$ of the logarithmic $1$-form  $\omega$,
we have a coordinate chart $(V(q),\psi)$ such that
\begin{equation}\label{equation:log_field}
\omega=\psi^*\left(\sum_{j=1}^{r(q)}\lambda_{j}\frac{dy_j}{y_j}\right).
\end{equation}
We can take a finite number of points  $q_{\beta}\in D$ with coordinate charts
$(V_{\beta},\psi_{\beta})$ satisfying (\ref{equation:log_field}) and such that the union  $\cup_{\beta}V_{\beta}$ covers  $D$.

Let  $U_i\subset N_i$ be open balls containing the singular points of $\omega$ in $X-D$
such that the diameter of $g_i(U_i)$ is smaller than  $\delta/10$. Using partition of unity we construct two $C^\infty$ complete real vector fields $u,v$ on $X$ such that :
\begin{itemize}
\item[(1)] their restrictions in $V_\beta$ satisfy that
\[D\psi_{\beta}(u)=\sum_{j=1}^{r(q_\beta)}\frac{y_j}{\lambda_j}\frac{\partial}{\partial y_j},\quad D\psi_{\beta}(v)=\sqrt{-1}\sum_{j=1}^{r(q_\beta)}\frac{y_j}{\lambda_j}\frac{\partial}{\partial y_j},\]
\item[(2)] at any point $p$ in $X-(\bigcup_{\beta}V_{\beta}\cup\bigcup_{i}U_i) $ they satisfy
$\omega(u_p)=1$, $\omega(v_p)=\sqrt{-1}$ and if $p$ belongs to
$T_i$ then these vector fields are tangent to $T_i$;
\item[(3)] their restriction to $\overline{U_i}$ vanishes in $p_i$.
\end{itemize}

The vector fields $u,v$ leave the divisor $D$ invariant. It follows that  the restriction of $u,v$ to $X-D$ are still complete vector fields and they satisfy  $\omega(u)=1,\omega(v)=\sqrt{-1}$ outside of $D\cup(\cup_{i}U_i)$.

Let $\tilde{u},\tilde{v}$ be the liftings of  $u,v$ with respect to  $Y$. Notice that  the vector
fields $\tilde{u},\tilde{v}$ are complete vector fields on  $Y$, which restricted to  $g^{-1}(U_b)-W(b)$ satisfy $\rho^{*}\omega(\tilde{u})=1,\rho^{*}\omega(\tilde{v})=\sqrt{-1}$.
It implies the existence of the diffeomorphism
\[
\begin{array}{lll}
\Phi:U_b\times (g^{-1}(b)-W(b)) & \rightarrow & g^{-1}(U_b)-W(b) \\
(t_{1}+b,t_{2}+b)\times\{q\}\quad  & \mapsto & \Phi_1(t_1,\Phi_2(t_{2},q)),
\end{array}
\]
where $\Phi_1,\Phi_2$ are flows of $\tilde{u},\tilde{v}$, respectively. The vector fields  $\tilde{u},\tilde{v}$
are tangent to  $\tilde{T}_j$ for every  $j\in J$. In particular, they are tangent to  $\cup_{j\in J(b)}\tilde{T}_{j}$. It follows that
\[
\Phi (U_b\times (g^{-1}(b)-W(b))\cap\overline{W(b)})=(g^{-1}(U_b)-W(b))-(g^{-1}(\partial U_b)\cap\overline{W(b)})
\]
as we wanted.
\end{proof}

\begin{example}\label{example:nonsingular}
Let  $\omega$ be a closed logarithmic $1$-form on  $\mathbb{P}^{n+1}$, with a simple normal crossing polar divisor $D=H_0+\cdots+H_k$ with $1\leq k\leq n+1$.
Let  $H_j$ be hyperplanes defined by $H_j=\{z_j=0\}$ where $[z_0:\cdots:z_{n+1}]$
are homogeneous coordinates for  $\mathbb{P}^{n+1}$. Take the universal covering
\[\begin{array}{lll}
\rho:\C^{n+1} & \rightarrow  & \quad\quad\mathbb{P}^{n+1}-D \\
{[1: x_1 :\cdots :x_{n+1}]} & \mapsto & {[1: e^{2\pi\sqrt{-1}x_1} :\cdots : e^{2\pi\sqrt{-1}x_{k}}: x_{k+1} : \cdots :x_{n+1} ]}.
\end{array}
\]
If we denote the  residues by $\mathrm{Res}(\omega,H_j)=\lambda_j$, then the pull-back $\rho^*\omega$ admits the following expression
 \[
 2\pi\sqrt{-1}\sum_{j=0}^k \lambda_jdx_j,
 \]
which is a linear 1-form on $\C^{n+1}$.
In this case, there are  no singularities outside the divisor
and the primitive  $g$ of $\rho^*\omega$ is a fibration of $\C^{n+1}$ with fiber $\C^n$.
In particular, the pair  $(\C^{n+1},g^{-1}(c))$ is $l$-connected for every $l$.
\end{example}

\begin{example}
Consider the closed rational $1$-form
\[
\omega=d\left(\frac{x^{2}+y^{2}+z^{2}}{xy}\right)
\]
in homogeneous coordinates  $[x:y:z]$ of $\mathbb{P}^2$.
The polar divisor $D$ of   $\omega$ has only two irreducible components $D_0=\{x=0\},D_1=\{y=0\}$, with $D=2D_0+2D_1$.
The singularities of  $\omega$ outside of  $D$ are the points $p_1=[1:1:0],p_2=[-1:1:0]$.

 The 1-form $\omega$ is exact in $\mathbb{P}^2-D$.
 The leaves of the foliation  $\mathcal{F}$ defined by $\omega$ in $\mathbb{P}^2-D$ coincide with
\[
\{x^2+y^2+z^2-\alpha xy=0\}-D\quad\text{with}\quad \alpha\in\C.
\]
If we assume that Proposition  \ref{proposition:locallyTrivial} is true in this situation, we would have for   $\delta>0$ sufficiently small a diffeomorphism
\[
\Phi:g^{-1}(B(2,\delta))-W(2)\cong B(2,\delta)\times \left(g^{-1}(2)-W(2)\right).
\]
But this is impossible since the set $g^{-1}(2)$ consists of two lines and the set $g^{-1}(2)-W(2)$ is not connected and the set
$g^{-1}(B(2,\delta))-W(2)$ is connected.

The construction of the vector field used to prove  Proposition  \ref{proposition:locallyTrivial} fails in this case, since at the singular
 points $q_1=[1:0:1],q_2=[-1:0:1],q_3=[0:1:1],q_4=[0:-1:1]$ the vector field
\[
u=\frac{x^2y}{x^2-y^2-z^2}\frac{\partial}{\partial x}+\frac{y^2x}{y^2-x^2-z^2}\frac{\partial}{\partial y}+\frac{2z}{xy}\frac{\partial}{\partial z}
\]
cannot be extended.
\end{example}

\subsection*{The proof of Theorem \ref{theorem:integralleaves}}

Define the following sets
\[
P(b,V)=g^{-1}(V)\cup W(b), R(b)=g^{-1}(U_b)-W(b), P^{R}(b,V)= P(b,V)-W(b),
\]
where $V$ is contained in $U_b$.

\begin{lemma}\label{lemma:fibern-connected}
Let  $V\subset U_b$ be a contractible subset.
If there exists a continuous map $\xi:U_b\times[0,1]\rightarrow U_b$ such that
 $\xi(y,0)=y$,and the sets $\xi(V\times[0,1]),\xi(U_b\times\{1\})$ lie in $V$. Then the pair $(g^{-1}(U_b),g^{-1}(V))$ is $n$-connected.
\end{lemma}
\begin{proof}
For each  $\tilde{T}_j$ with  $j\in J(b)$, we can choose a vector field
 $\nu_j$ tangent to the level sets of $g$ and pointing to the interior of  $W(b)$. The vector field $\nu$ on $\partial W(b)$ defined by $\nu|_{\tilde{T}_j}+\nu_j$ allows us
 to construct a deformation  $h:W(b)\times[0,1]\rightarrow W(b)$ such that  $h(y,0)=y$, and the image of $h(W(b)\times\{1\})=W'(b)$ has empty intersection with $R(b)$.

The map  $h(y,1-t)$ gives us a deformation of the pair  $(g^{-1}(U_b)-W'(b),P(b,V)-W'(b))$ to the pair $(R(b),P^{R}(b,V))$.
By \cite[5.5 Excision II]{simpson1993lefschetz}, the pairs  $(g^{-1}(U_b),P(b,V))$ and $(R(b),P^{R}(b,V))$ have the same $l$-connectivity.
Therefore,  Proposition \ref{proposition:locallyTrivial}.
implies that the pair $(g^{-1}(U_b),P(b,V))$ is $l$-connected for every $l$. 

Now, consider the pair $(\tilde{N}_{j}^{\circ}\cap g^{-1}(U_b),\tilde{N}_{j}^{\circ}\cap g^{-1}(V))=(U_{b,j},V_j)$ with $j\in J(b)$.
Since  $V$ is contractible and the restriction of  $g$ to $\tilde{N}_{j}-\tilde{F}_j$ is a trivial fibration,  Lemma \ref{lemma:singularTheory} and Property \cite[5.3 Deformation]{simpson1993lefschetz} imply that the pair $(U_{b,j},V_j)$ is $l$-connected  for every  $l$ if $\tilde{F}_j\subset V_j$, and $n$-connected if $\tilde{F}_j$ is not contained in $V_j$.
 Therefore the pair $(W(b),\cup_{j\in J(b)}V_j)$ is at least $n$-connected.

The vector field  $-\nu$ points toward the interior of  $R(b)$.
Analogously, we define a deformation  $h':R(b)\times[0,1]\rightarrow R(b)$ such that the closure of the image  $h'(R(b)\times\{1\})=R'(b)$ has empty intersection with
$\overline{W(b)}$. Considering the set $R'(b)\cap g^{-1}(V)$ and the tangency of $\nu$ to the level sets of $g$,
 \cite[5.5 Excision II]{simpson1993lefschetz} implies that the pair  $(P(b,V),g^{-1}(V))$ is $n$-connected. Since $(g^{-1}(U_b),P(b,V))$ is $l$-connected for every $l$, by \cite[5.2 Transitivity]{simpson1993lefschetz}  the pair $(g^{-1}(U_b),g^{-1}(V))$ is $n$-connected.
\end{proof}

\begin{proof}[{\bf Proof of Theorem \ref{theorem:integralleaves}}]
Take a  triangulation $\Delta$ of  $\C$ by equilateral triangles with sides of length $\delta$, such that one of the vertices in $\mathrm{V}_\Delta$ is  $c\in\C$.
Let $H_l$ be the family of concentric hexagons with center $c$ and vertices in $\mathrm{V}_\Delta$. Label by $c_i$ the vertices  $\mathrm{V}_\Delta$ such that
 between  $c_i$ and $c_{i+1}$ there always exists an edge  $e_i\in\mathrm{E}_\Delta$ of the triangulation  $\Delta$, and $c_0=c$.
Also, the vertices $c_i$ with $6(l-1)l/2<i\leq 6l(l+1)/2$ are in the hexagon  $H_l$.

Consider the open sets  $U_i=B_{\C}(c_i,\delta)$ and $W_i=\cup_{j\leq i}U_j$. Since the intersection
 $U_i\cap W_{i-1}=V_i$ is contractible in $U_i$,  Lemma \ref{lemma:fibern-connected} gives that the pair  $(g^{-1}(U_i),g^{-1}(V_i))$ is $n$-connected.

The  $W_i$-closures of the sets  $(W_i-W_{i-1})$ and $(W_{i-1}-U_i)$ are disjoint, thus the previous paragraph combined with \cite[5.4 Excision]{simpson1993lefschetz} imply that the
pair $(g^{-1}(W_i),g^{-1}(W_{i-1}))$ is $n$-connected  for every $i$. From \cite[5.2 Transitivity]{simpson1993lefschetz} we  deduce that the pair $(g^{-1}(W_i),g^{-1}(W_{0}))$ is $n$-connected for every $i$. Taking  $V=c$ in  Lemma \ref{lemma:fibern-connected}, we see that  $(g^{-1}(W_0),g^{-1}(c))$ is $n$-connected. Hence $(g^{-1}(W_i),g^{-1}(c))$ is $n$-connected for all $i$.
 As any representative element of a class in $\pi_l(Y,g^{-1}(c))$ is contained in some pair $(g^{-1}(W_i),g^{-1}(c))$ we conclude that the pair $(Y,g^{-1}(c))$ is $n$-connected.
\end{proof}

\bibliographystyle{alpha}

\end{document}